\documentclass
[12pt]{article}

\usepackage{
amsthm,amssymb,xcolor,hyperref,graphicx}

\usepackage[leqno]{amsmath}
\usepackage{tikz}
\usepackage{todonotes}

\theoremstyle{plain}
\newtheorem{theorem}{Theorem}[section]

\newtheorem{observation}[theorem]{Observation}
\newtheorem{corollary}[theorem]{Corollary}
\newtheorem{proposition}[theorem]{Proposition}

\newtheorem{problem}[theorem]{Problem}
\newtheorem{remark}[theorem]{Remark}
\newtheorem{definition}[theorem]{Definition}
\usepackage{latexsym}
\usepackage{amsfonts}
\usepackage{float}

\def\lk{\mathit{lk}}

\title{Vertex spanning planar Laman graphs in triangulated surfaces}

\author{Eran Nevo \thanks{
Einstein Institute of Mathematics,
The Hebrew University of Jerusalem, Jerusalem, 91904 Israel.\href
{mailto:nevo@math.huji.ac.il}
{nevo@math.huji.ac.il}. Eran Nevo was partially supported by the Israel Science Foundation grants ISF-1695/15 and ISF-2480/20 and by ISF-BSF joint grant 2016288.}
 \and
 Simion Tarabykin \thanks{
 Einstein Institute of Mathematics,
The Hebrew University of Jerusalem, Jerusalem, 91904 Israel.
\href
 {mailto:simon.trabykin@gmail.com}
 {simon.trabykin@gmail.com}. Simion Tarabykin was partially supported by ISF grant 1695/15.}
}

\date{\today}
\begin{document}
\maketitle
\begin{abstract}
We prove that every triangulation of either of the torus, projective plane and Klein bottle,  contains a vertex-spanning planar Laman graph as a subcomplex. Invoking a result of Kir{\'a}ly, we conclude that every $1$-skeleton of a triangulation of a surface of nonnegative Euler characteristic has a rigid realization in the plane using at most 26 locations for the vertices.
\end{abstract}

\section{Introduction}
A basic object of study in Framework Rigidity is a graph $G=(V,E)$ made of bars and joints in Eucledian space $\mathbb{R}^d$. Such a realization is given by specifying a map $p:V\rightarrow \mathbb{R}^d$. The pair $(G,p)$ is called a \emph{framework}. It is important, for both mathematicians and engineers, to know whether the framework $(G,p)$ is \emph{infinitesimaly rigid}, namely, whether every small enough perturbation of $p$ that preserves all the edge lengths, up to first order, is the restriction to $V$ of some rigid motion of the entire space $\mathbb{R}^d$. A graph $G$ that admits an infinitesimally rigid framework $(G,p)$ is called \emph{$d$-rigid}. If $G$ is $d$-rigid, a generic map $p:V\rightarrow \mathbb{R}^d$ makes $(G,p)$ infinitesimally rigid.

The following question arises: for $G$ $d$-rigid, how small can a subset $A\subseteq \mathbb{R}^d$ be, such that there exists  an infinitesimaly rigid framework $(G,p)$ with $p:V\rightarrow A$? Likewise for a family $F$ of $d$-rigid finite graphs:
Denote by
$c_d(F)$ the minimum cardinality $|A|$ over subsets $A\subseteq \mathbb{R}^d$ satisfying that for every graph $G\in F$ there exists $p:V(G)\rightarrow A$ such that $(G,p)$ is infinitesimally rigid.

Jordan and Fekete~\cite{Fekete-Jordan} showed that for the family $F_1$ of $1$-rigid graphs, namely the connected graphs, $c_1(F_1)=2$, and for $d\ge 2$, the family $F_d$ of $d$-rigid graphs has  $c_d(F_d)=\infty$, namely no such finite $A$ exists; see also~\cite{Adiprasito-Nevo2020}. Let us restrict to the subfamily $F(g)\subseteq F_3$, of $1$-skeleta of triangulations of the surface of genus $g$ (orientable or not). Indeed, a fundamental result of Fogelsanger~\cite{Fogelsanger} asserts that for every $g$, every graph $G\in F(g)$ is $3$-rigid. Adiprasito and Nevo~\cite{Adiprasito-Nevo2020} showed that $c_3(F(g))$ is finite for any fixed genus $g$, and asked whether there exists an absolute constant $c$ such that $c_3(F(g))\le c$ for all $g$. The same question can be asked in the plane:
\begin{problem}\label{prob:c_2(F(g))<c?}
Does there exist an absolute constant $c$ such that $c_2(F(g))\le c$ for all genus $g$?
\end{problem}
Kir{\'a}ly~\cite{Kiraly} showed that $c_2(F(g))=O(\sqrt{g})$. He also
proved that for the family $F(PL)$ of planar Laman graphs, $c_2(F(PL))$ is finite (in fact, at most $26$), answering a question of Whiteley~\cite[Prob.6.4]{Adiprasito-Nevo2020}.
Thus, an answer Yes to Problem~\ref{prob:c_2(F(g))<c?} would follow from an answer Yes to the following problem:
\begin{problem}\label{prob:spanning}
Does every triangulation of a surface (compact, connected, without boundary) admit a vertex spanning planar Laman graph?
\end{problem}
As mentioned, Kir{\'a}ly showed this for the $2$-sphere, denoted $S^2$.
We answer Problem~\ref{prob:spanning} in the affirmative for the surfaces of nonnegative Euler characteristic, by proving a stronger structural-topological result:
\begin{theorem}[Main Theorem]\label{thm:Main} The following holds:

(i) Every triangulation of the projective plane $\mathbb{R}P^2$ contains a vertex spanning disc (as a subcomplex).

(ii) Every triangulation of the Torus $T$  contains a vertex spanning cylinder.

(iii) Every triangulation of the Klein bottle $K$ contains a vertex spanning, planar,
2-dimensional complex; it is either a cylinder,
or a connected sum of two triangulated discs along a triangle\footnote{This connected sum may have at most two edges contained in no triangle face;
deleting them yields a pure complex which is strongly-connected.}.
\end{theorem}
For a topological space $M$, denote by $F(M)$ the family of $1$-skeleta of triangulations of $M$.
\begin{corollary}\label{cor:Main}
For $M\in \{T, K, \mathbb{R}P^2, S^2\}$, $c_2(M)\leq 26$.
\end{corollary}
To see that Theorem~\ref{thm:Main} implies Corollary~\ref{cor:Main}, note two facts: (i) all the vertex spanning subcomplexes in  Theorem~\ref{thm:Main} have a $2$-rigid graph (indeed, clearly all strongly-connected pure $d$-dimensional simplicial complexes have a $d$-rigid $1$-skeleton, see e.g.~\cite[Lem.6.2]{Kalai:LBT}), thus each of them contains a minimal $2$-rigid, namely Laman, spanning subgraph; and (ii) these Laman graphs are planar -- this is clear as the $2$-dimensional subcomplexes containing them are themself planar.
Now apply
Kir{\'a}ly's result that $c_2(F(PL))\le 26$~\cite{Kiraly}.

The basic idea in the proof of Theorem~\ref{thm:Main} is to use induction over \emph{vertex splits}:
first we find a suitable spanning subsurface in each \emph{irreducible triangulation} of $M\in \{T, K, \mathbb{R}P^2\}$, and then \emph{extend} the spanning subsurface along vertex splits. One has to be careful to extend at each vertex split in such a way that  the new spanning subsurface is again \emph{extendible}. This is defined and explained in Section~\ref{sec:split}, see the Extension Theorem~\ref{thm:extention}.

\textbf{Outline.} In Section~\ref{sec:prelim} we give the necessary background on rigidity and on irreducible triangulations, in Section~\ref{sec:split} we prove the Main Theorem~\ref{thm:Main} via proving the Extension Theorem~\ref{thm:extention}, modulo a lemma on rearranging the vertex splits, proved in Section~\ref{sec:rearrange}.
We end with concluding remarks in Section~\ref{sec:conclude}.

\section{Preliminaries}\label{sec:prelim}
\subsection{Rigidity}
Let $G=(V,E)$ be a graph, and let $p:V\rightarrow \mathbb{R}^d$ be a map.
An \emph{infinitesimal motion} of the framework $(G,p)$ is a map $a:V\rightarrow \mathbb{R}^d$ (think of $a$ as an assignment of velocity vectors) such that for all edges $vu\in V$, the following inner product vanishes:
\[ \langle a(v)-a(u), p(v)-p(u)\rangle =0.
\]
The motion $a$ is \emph{trivial} if the relation above is satisfied for every pair of vertices in $V$; otherwise $a$ is \emph{nontrivial}. The framework $(G,p)$ is \emph{infinitesimally rigid} if all its motions are trivial. This definition is equivalent to the one given in the introduction. A graph $G$ admitting such an infinitesimally rigid framework $(G,p)$ is \emph{$d$-rigid}. In that case the subset of maps $p$ such that $(G,p)$ is infinitesimally rigid is Zariski dense in the space $\mathbb{R}^{d|V(G)|}$ of all maps $q:V(G)\rightarrow \mathbb{R}^{d}$. The readers may consult e.g.~\cite{Connelly:RigiditySurvey, graver1993combinatorial} for further background on rigidity.


\subsection{Irreducible triangulations}
Let $vu$ be an edge in a graph $G=(V,E)$. Contract $v$ to $u$ to obtain the graph $G'=(V-v,E')$, so $E'=(E\setminus \{wv:\ wv\in E\})\cup \{wu:\ wv\in E, \ w\neq u\}$. This operation is called an \emph{edge contraction} at $vu$.
The inverse operation, that starts with $G'$ and produces $G$ is called a \emph{vertex split} at $u$.
Similarly one defines edge contraction (and vertex split) for simplicial complexes: replace the faces of the form $F\cup\{v\}$ in a simplicial complex $\Delta$ ($v\notin F$) by $F\cup\{u\}$ (and remove duplicates if they appear) to obtain a new simplicial complex $\Delta'$.

A triangulation $\Delta$ of the surface of genus $g$, $M_g$, is \emph{irreducible} if each contraction of an edge of $\Delta$ changes the tolopogy; equivalently, the following combinatorial condition holds: each edge in $\Delta$ belongs to an \emph{empty} triangle $F$ of $\Delta$, namely, $F\notin \Delta$ and its boundary complex $\partial F \subseteq \Delta$. Barnette and Edelson showed:

\begin{theorem}\label{thm:irr-triang}\cite{Barnette-Edelson-NonOrientable, Barnette-Edelson-Orientable}
For all $g$, $M_g$ has finitely many irreducible triangulations.
 \end{theorem}
When the Euler characteristic $\chi(M_g)\geq 0$, the irreducible triangulations were characterized (up to combinatorial isomorphism) in a series of works:
the $2$-sphere has a unique irreducible triangulation, namely the boundary of the tetrahedron, see e.g. Whiteley~\cite{Whiteley-split} for a proof that all maximal planar graphs are $3$-rigid using this fact; $\mathbb{R}P^2$ has
 two irreducible triangulations, see Barnette~\cite{Barnette1982GeneratingTT}; the torus has
21, see Lavrenchenko~\cite{Lavrenchenko1990IrreducibleTO};
the Klein bottle has 29, see Lavrenchenko-Negami~\cite{Lawrencenko1997IrreducibleTO} and a correction by Sulanke~\cite{Sulanke2006NoteOT}.

We will use these characterizations in the proof of Theorem~\ref{thm:Main}, in the next section.
Note that every triangulation of $M_g$ can be obtained from some irreducible triangulation of $M_g$ by a sequence of vertex splits; each intermediate complex also triangulates $M_g$.

Specifically, let $\Delta$ be a triangulation of $M_g$. A \emph{vertex split} $\Delta\rightarrow \Delta'$ at $v$ is given as follows: decompose the link of $v$, $\lk_v(\Delta)$, which is a cycle,  into two closed intervals in the cyclic order, $[x_1,x_k]$ and $[x_k,x_1]$, where $k\neq 1$. Now replace the star of $v$ in $\Delta$ by the union of two cones over cycles  $v*(x_1,\ldots,x_k,v',x_1)\cup v'*(x_k,\ldots,x_1,v,x_k)$ (where $v'$ is a new vertex) to obtain a triangulation $\Delta'$ of $M_g$.

\section{Extentions}\label{sec:split}
\subsection{Extendible subsurfaces}
Informally, we want the vertex split on a triangulation of $M_g$, $\Delta\rightarrow \Delta'$, to allow an \emph{extension} $S'\subseteq \Delta'$ of the spanning disc/cylinder/etc  $S\subseteq \Delta$; see Theorem~\ref{thm:Main} for the relevant topology of $S$. Formally, for a triangulation $\Delta$ of the surface $M_g$, define:

\begin{definition}\label{def:extendible}
A vertex spanning subsurface (with boundary) $S\subseteq \Delta$ is \emph{extendible} if for every vertex split $\Delta\rightarrow \Delta'$
there exists a subsurface $S'\subseteq \Delta'$ such that either

(i) $S'$ is obtained from $S$ by a vertex split at the same vertex, denote it by $v$.
Namely, the contraction of the edge $vv'$ in $S'$ results in $S$; or

(ii) $S'$ is obtained from $S$ by adding a cone over an interval in the boundary of $S$. (The apex of the cone is the new vertex, in particular not in $S$.)
\end{definition}
Note that such $S'\subseteq \Delta'$ is vertex spanning and homeomorphic to $S$.

\begin{theorem}[Extension Theorem]\label{thm:extention}
Let $\Delta$ triangulate some surface $M_g$ (compact, connected, without boundary), and let $S\subseteq \Delta$ be a vertex
spanning subsurface. Then:

(1) $S$ is extendible in $\Delta$ iff it contains at least one edge
from every triangle in $\Delta$.

(2) Let $\Delta\rightarrow \Delta'$ be a vertex split.
If $S$ is extendible then it has an extendible extension $S'\subseteq \Delta'$.
\end{theorem}
In all the irreducible triangulations of $\mathbb{R}P^2$, $T$ and $K$, except for the four so called cross-cap triangulations of the Klein bottle $K$ (see Figure~\ref{fig:4crosscap}), we find an extendible spanning subsurface $S$ to start with: for $\mathbb{R}P^2$ the subsurface $S$ is a disc, and for $T$ and $K$ the subsurface $S$ is a cylinder. See Figure~\ref{fig:SinRP2} for spanning discs in the two irreducible triangulation of $\mathbb{R}P^2$, and see Figure~\ref{fig:SinT} (resp. Figure~\ref{fig:SinK}) for spanning cylinders in the irreducible triangulations of the torus (resp. the non-crosscap irreducible triangulations of Klein's bottle).

\begin{figure}[H]
  \centering
  \includegraphics[width=\textwidth]{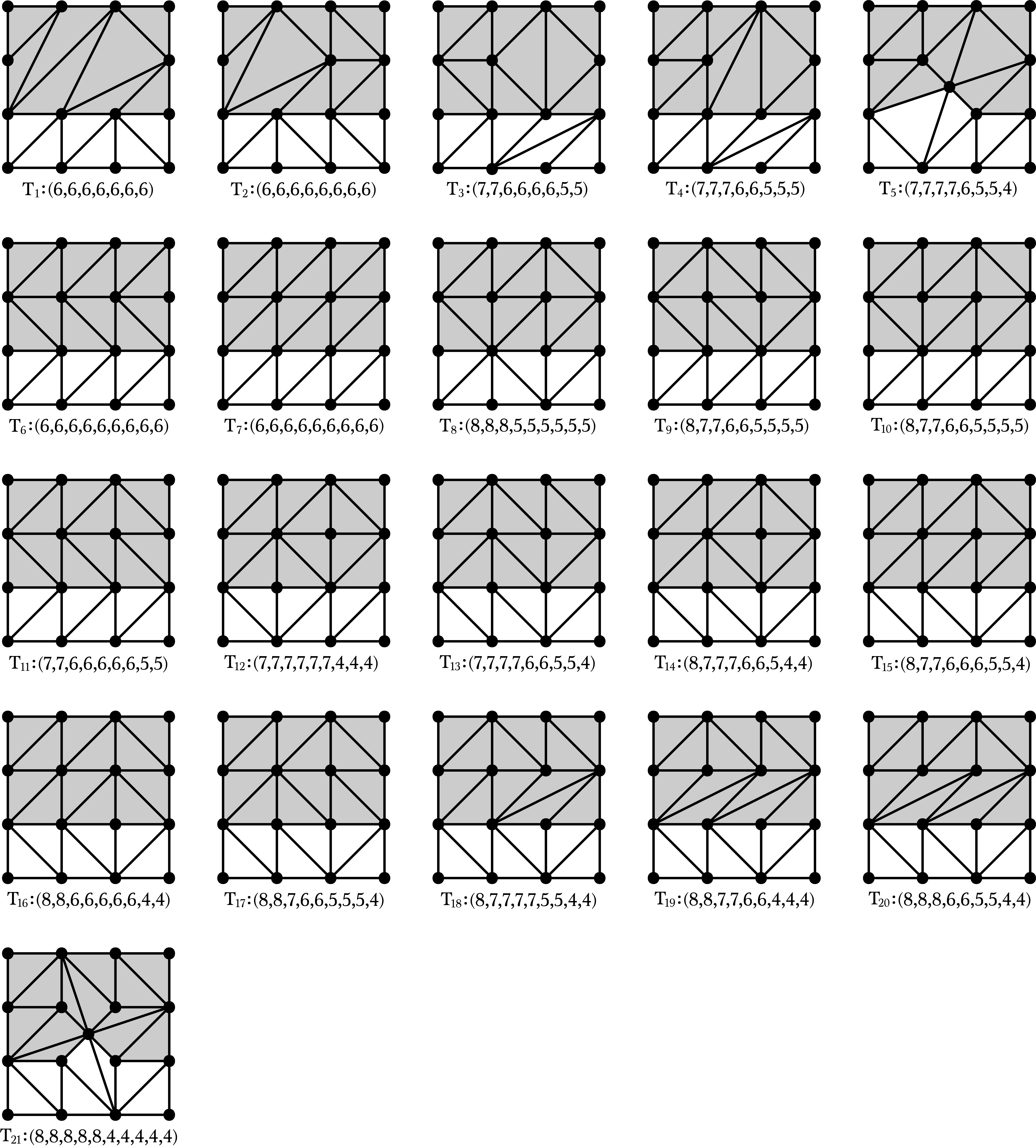}
  \caption{Spanning cylinders, in grey,  in the 21 irreducible triangulations of the torus.}\label{fig:SinT}
\end{figure}

\begin{figure}[H]
  \centering
  \includegraphics[width=\textwidth]{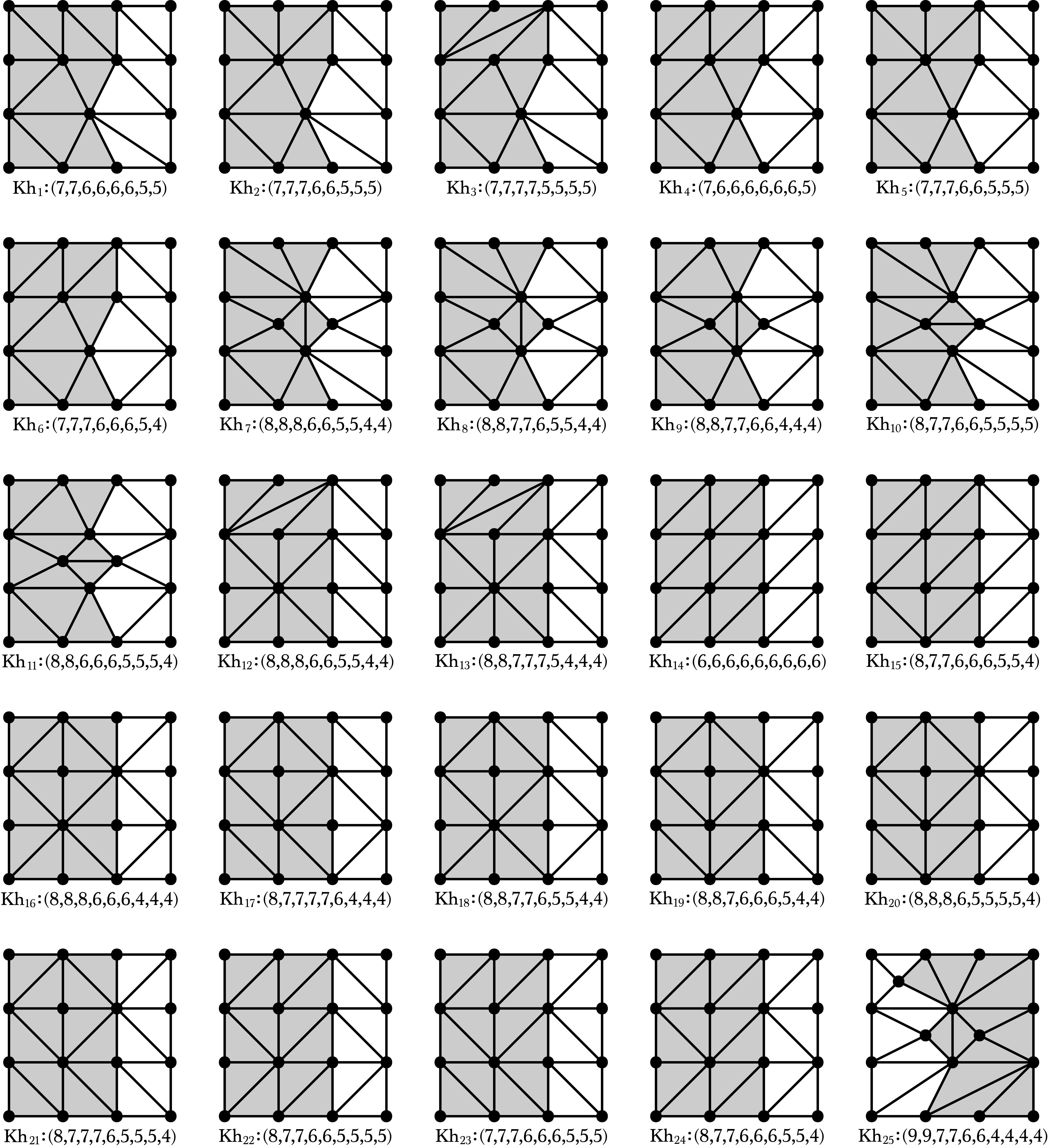}
  \caption{Spanning cylinders, in grey,  in the 25 non-crosscap irreducible triangulations of Klein's bottle.}\label{fig:SinK}
\end{figure}
\begin{figure}[H]
  \centering
  \includegraphics[width=0.6\textwidth]{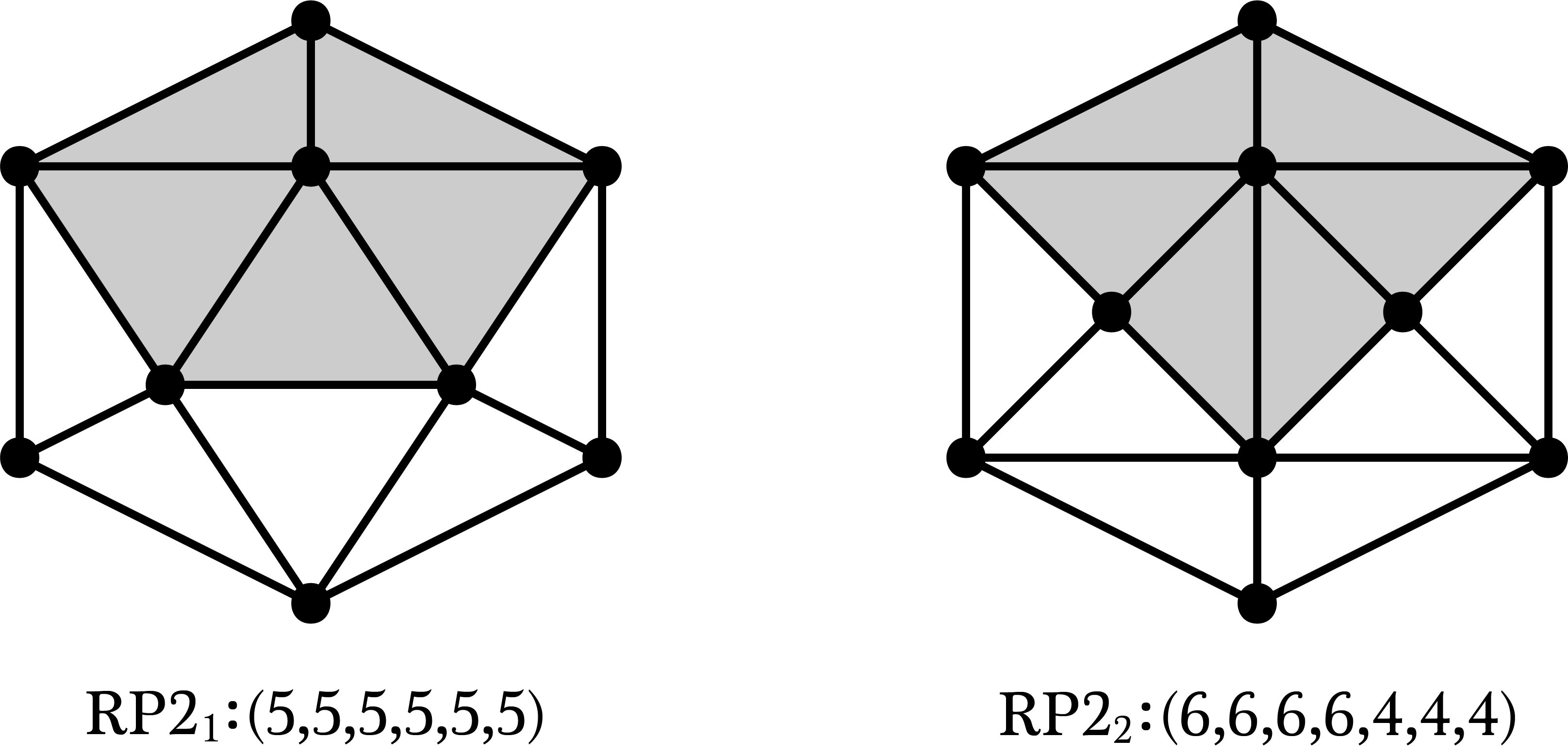}
  \caption{Spanning discs, in grey, in the irreducible triangulations of the projective plane.}\label{fig:SinRP2}
\end{figure}

For a proof of the Extension Theorem~\ref{thm:extention},
the ``only if" part in (1) is easy, see Figure~\ref{fig:only_if} for an illustration.
Indeed, assume by contradiction that $S$ contains no edge of the triangle $xyz\in\Delta$. Perform the vertex split at $x$ such that the new vertex $x'$ has $x,y,z$ as its only neighbors in $\Delta'$. Then all edges in $S'$ that are not in $S$ contain $x'$, hence non of the edges $xy,yz,xz$ is in $S'$. But then $S'$ contains no triangle containing $x'$, a contradiction.

\begin{figure}[H]
  \centering
  \includegraphics[width=\textwidth]{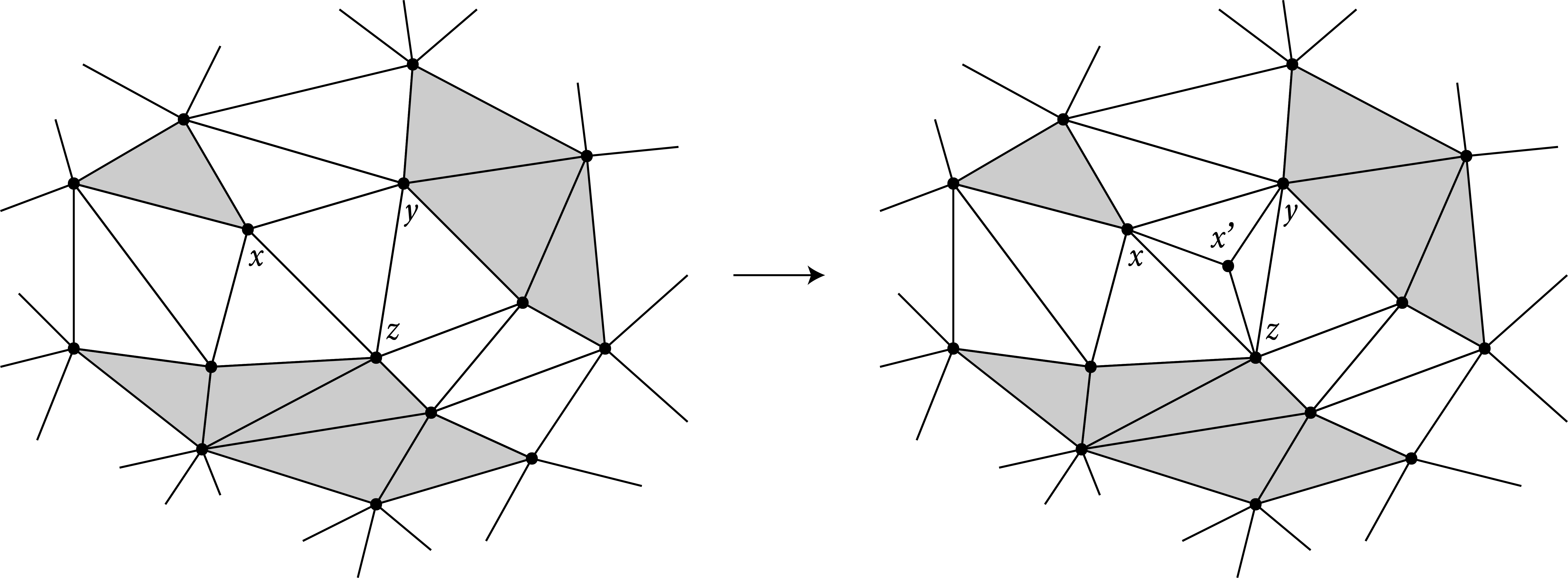}
  \caption{Non extendible subcomplex. On the left, the grey subcomplex is not extendible, as the vertex split with new vertex $x'$ on the right demonstrates.}\label{fig:only_if}
\end{figure}
However, it is the ``if" part that is important to us.
We detail below, and illustrate in figures, how to choose the extendible subsurface $S'$, according to the intersection of the closed star of the split vertex $v$ with $S$ -- either by a vertex split, see Figure~\ref{fig:extend-via-split}, or by coning over a boundary interval, see Figure~\ref{fig:extend-via-coning} for illustrations.


\begin{figure}[H]
  \centering
  \includegraphics[scale=3, width=\textwidth]{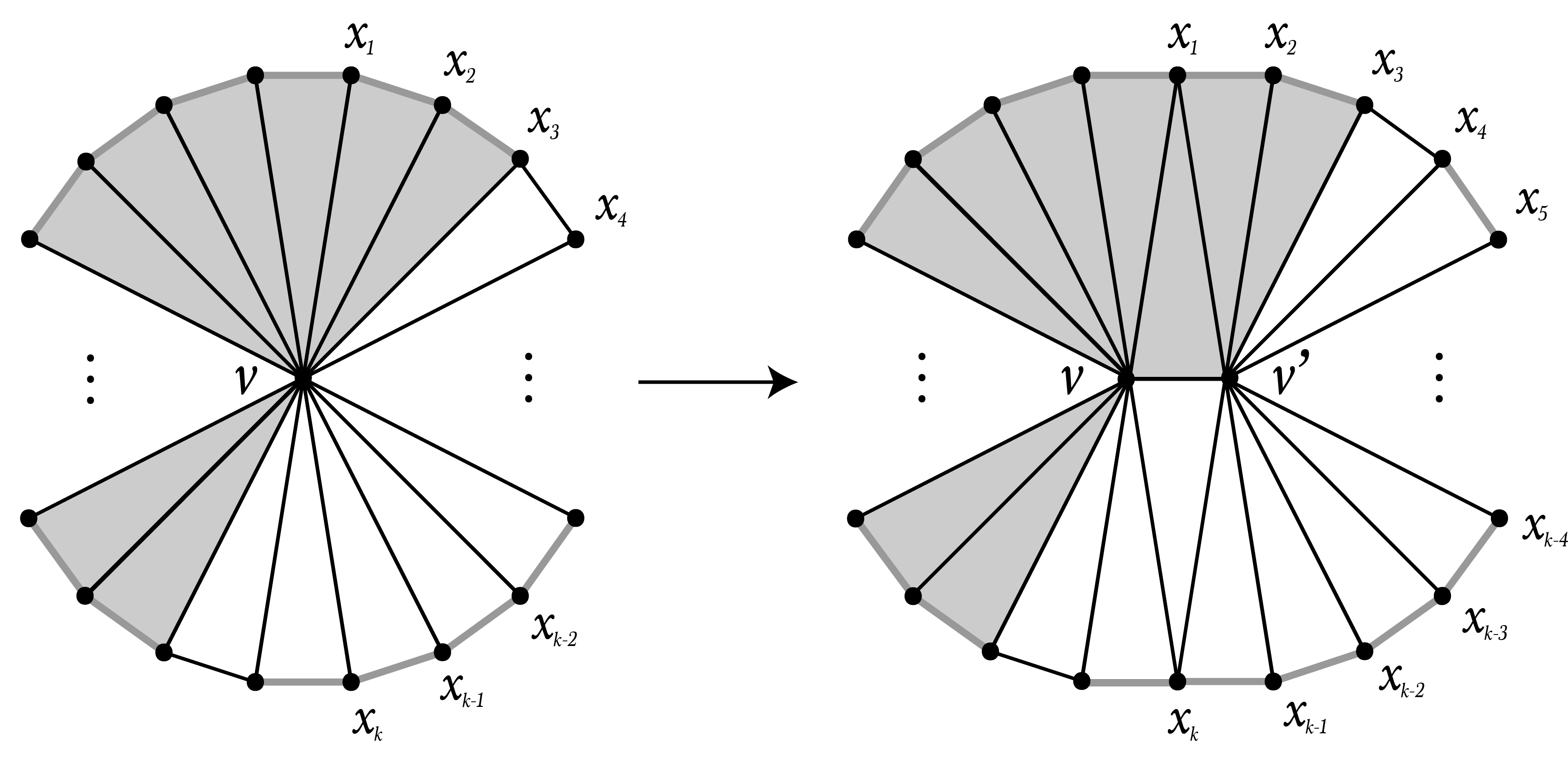}
  
  \caption{Extension $S\rightarrow S'$ via a vertex split. The grey triangles and grey edges are in $S$.}\label{fig:extend-via-split}
\end{figure}

\begin{figure}[H]
  \centering
  \includegraphics[width=\textwidth]{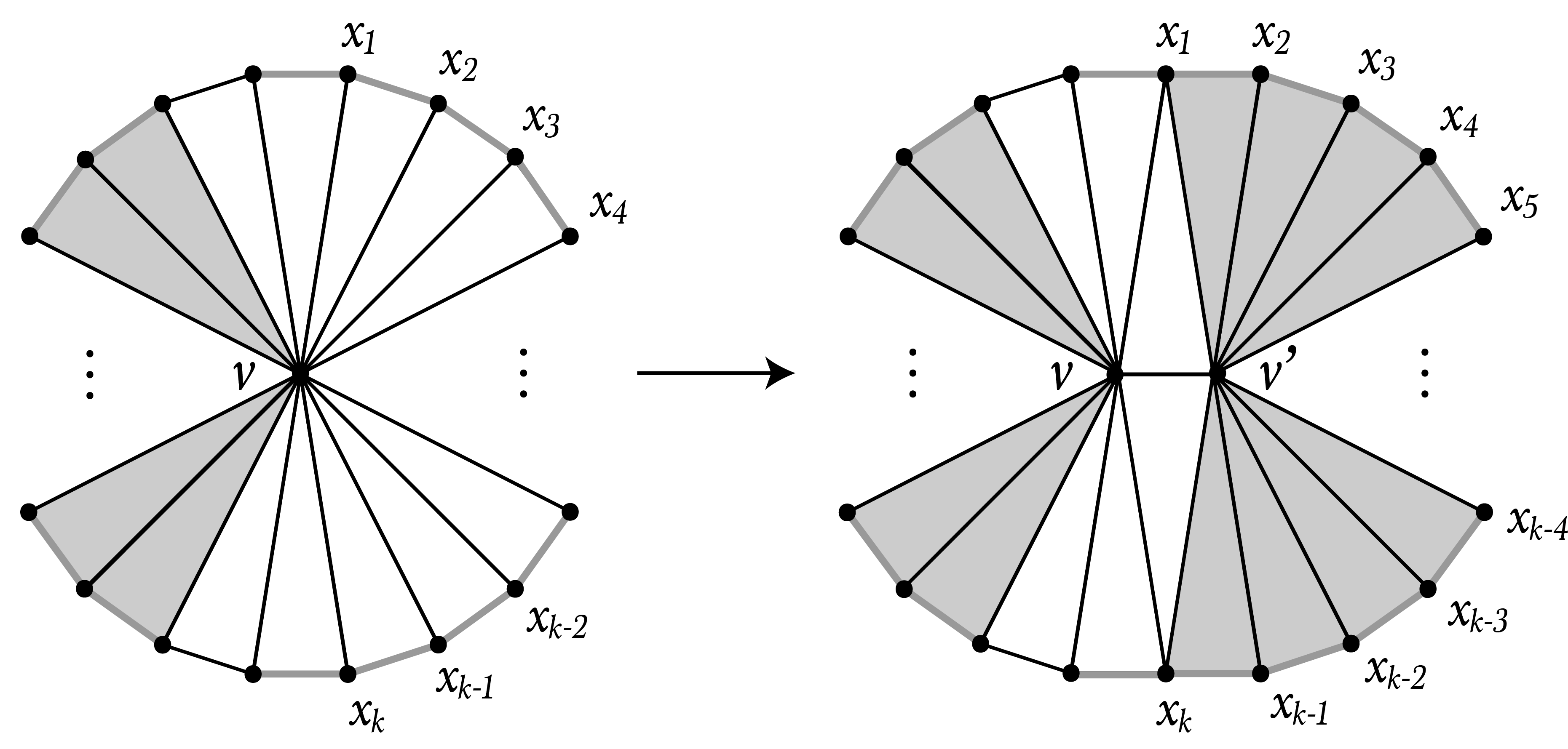}
  \caption{Extension $S\rightarrow S'$ via coning over a boundary interval in the spanning subsurface $S$. The grey triangles and grey edges are in $S$.}\label{fig:extend-via-coning}
\end{figure}

All cases are treated similarly. Let us  describe an exhaustive list of cases:  the intersection of $S$ with the star of $v$, consists of (i) a subcomplex $C$ of consecutive triangles, each of them contains $v$ (this collection is nonempty as $v\in S$, and it may exhaust the star), and of (ii) a path $P$ in the link of $v$, denoted $\lk_v(\Delta)$, such that $C\cup P$ contains all the vertices in the star of $v$ (as $S$ contains an edge from each triangle and is vertex spanning, and $P$ may be empty).
Now the vertex split at $v$ introduces a new vertex $v'\in \Delta'$, and there are exactly two common neighbors $x_1$ and $x_k$ of $v$ and $v'$ in $\Delta'$, both belong to $\lk_v(\Delta)$, and the indices are according to the cyclic order along the link of $v$. We distinguish the following cases:

(1) If $x_1$ and $x_k$ are in $P$
then we cone with $v'$ over the boundary interval $[x_1,x_k]$ in $\lk_v(\Delta)$, as in Figure~\ref{fig:extend-via-coning}, to obtain $S'$ from $S$.
We are left to check that every triangle in $\Delta'$ contains an edge in $S'$. This is clear for all triangles not containing $v'$. For triangles $v'x_ix_{i+1}$ where $1\le i\le k-1$ it follows by the coning construction.
Only two triangles remain: $vv'x_1$ and $vv'x_k$. The first contains $v'x_1$ and the second contains $v'x_1$; both of these edges are in $S'$.

(2) Else we find a suitable vertex split, chosen according to which of $x_1$ and $x_k$ are in $C$, as demonstrated in Figure~\ref{fig:extend-via-split}.
We detail these cases.

Case (2a): $x_1,x_k\in C$. Denote the intersection of $C$ with the link of $v$, which is an interval, by $[x_{k-s},x_{1+t}]$ in the cyclic order, where $s,t\ge 0$.
Then the vertex split at $v$ in $S$ creates exactly the following triangles containing $v'$: $vv'x_1$, $vv'x_k$, and $v'x_i x_{i+1}$ for $i\in [1,t]\cup [k-s,k-1]$. All the other triangles in $\Delta'$ containing $v'$ have an edge in $P\cup\{v'x_{1+t},v'x_{k-s}\}$. We conclude that every triangle in $\Delta'$ contains an edge in $S'$.

Case (2b): $x_1\in C,\  x_k\in P\setminus C$.
(The last case, (2c): $x_k\in C,\  x_1\in P\setminus C$ is treated similarly by symmetry.) Denote the intersection of $C$ with $\lk_v(\Delta)$, which is an interval, by $[x_{k+s},x_{1+t}]$ in the cyclic order, where $s\ge 1,\ t\ge 0$.
Then the vertex split at $v$ in $S$ creates exactly the following triangles containing $v'$: $vv'x_1$, and $v'x_i x_{i+1}$ for $i\in [1,t]$. All the other triangles in $\Delta'$ containing $v'$ has an edge in $P\cup\{v'x_{1+t},v'v\}$. We conclude that every triangle in $\Delta'$ contains an edge in $S'$.
This concludes the proof of the Extension Theorem~\ref{thm:extention}. $\square$

\subsection{Extension for the cross-cap triangulations of the Klein bottle}
To complete the proof of Theorem~\ref{thm:Main} we are left to deal with vertex splits over the four cross-cap triangulations of $K$.
In each of the four, $ABC$ is a noncontractible cycle; see Figure~\ref{fig:4crosscap}.

\begin{figure}[H]
  \centering
  \includegraphics[page=1,width=0.8\textwidth]{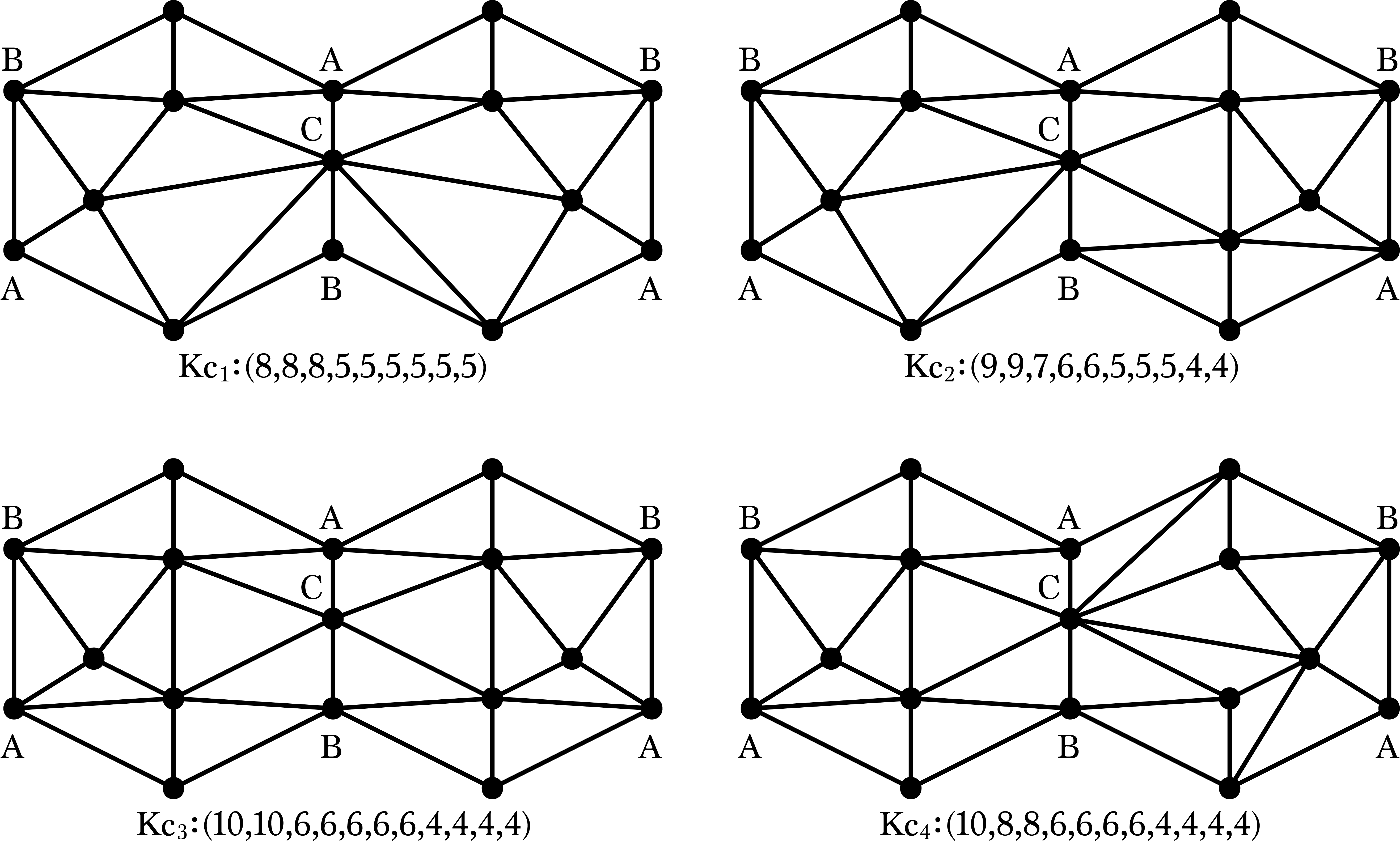}
  \caption{The four cross-cap irreducible triangulations of $K$. The cycle $ABC$ is  noncontractible in each of them.}\label{fig:4crosscap}
\end{figure}
Notice that each of these four triangulations is a connected sum of two $\mathbb{R}P^2$'s along the triangle $ABC$; this triangle can be part of the spanning disc in each $\mathbb{R}P^2$.

Our analysis depends on whether $ABC$ survives the vertex splits or not; namely, on whether the restriction of the vertex splits to $ABC$ is always a $3$-cycle or that for some split it becomes a $4$-cycle. Formally,  split the vertex $A$ (similarly for vertices $B$ and $C$) into two new vertices $A'$ and $A''$, with two common neighbors $x_1$ and $x_k$ numbered in the cyclic order on $\lk_A(\Delta)$, and where $A'x_2$ and $A''x_{k+1}$ are edges in $\Delta'$.
If both $B$ and $C$ are in the interval $[x_1,x_k]$, then the induced complex on $A'BC$ is a $3$-cycle and we name $A'=A$ and say that $ABC$ \emph{survived} the split. Else\footnote{In case $x_1=B$ and $x_k=C$ both $A'BC$ and $A''BC$ are $3$-cycles, and we chose to name $A'=A$ rather than $A''=A$.}, if both $B$ and $C$ are in the complementary closed interval $[x_k,x_1]$,
then the induced complex on $A''BC$ is a $3$-cycle and we name $A''=A$ and say that $ABC$ \emph{survived} the split. Else, the induced complex on $A'A''BC$ is a $4$-cycle and we say that $ABC$ did \emph{not survive} the split.

\textbf{Case 1: $ABC$ survives the vertex splits.} Then the resulted triangulation of $K$ is again a connected sum of two $\mathbb{R}P^2$'s along $ABC$, and $ABC$ can be taken as a triangle in each of the two spanning discs of the two $\mathbb{R}P^2$'s, e.g. by taking the star of $B$ as the spanning disc in either of the two irreducible triangulations of $\mathbb{R}P^2$ (according to the vertex labels in Figure~\ref{fig:4crosscap}) and extend from there along vertex splits.
Then, the connected sum of those two discs is a vertex spanning planar
simplicial complex.

\textbf{Case 2: $ABC$ does not survive the vertex splits.}
Then some vertex split induced also a vertex split of the cycle $ABC$, making it a $4$-cycle.
We will show in Proposition~\ref{prop:ABCsplitsFirst} a reduction to the case where the \emph{first} vertex split induces a split on $ABC$, making it a $4$-cycle.

If $C$ splits first (resp. $A$ or $B$), choose $S$ to be a spanning pinched disc at $C$ (resp. $A$ or $B$); see Figure~\ref{fig:pinched_disc}
(resp. Figure~\ref{fig:pinched_disc-A} or
Figure~\ref{fig:pinched_disc-B}).
\begin{figure}[H]
  \centering
\includegraphics[width=0.7\textwidth]{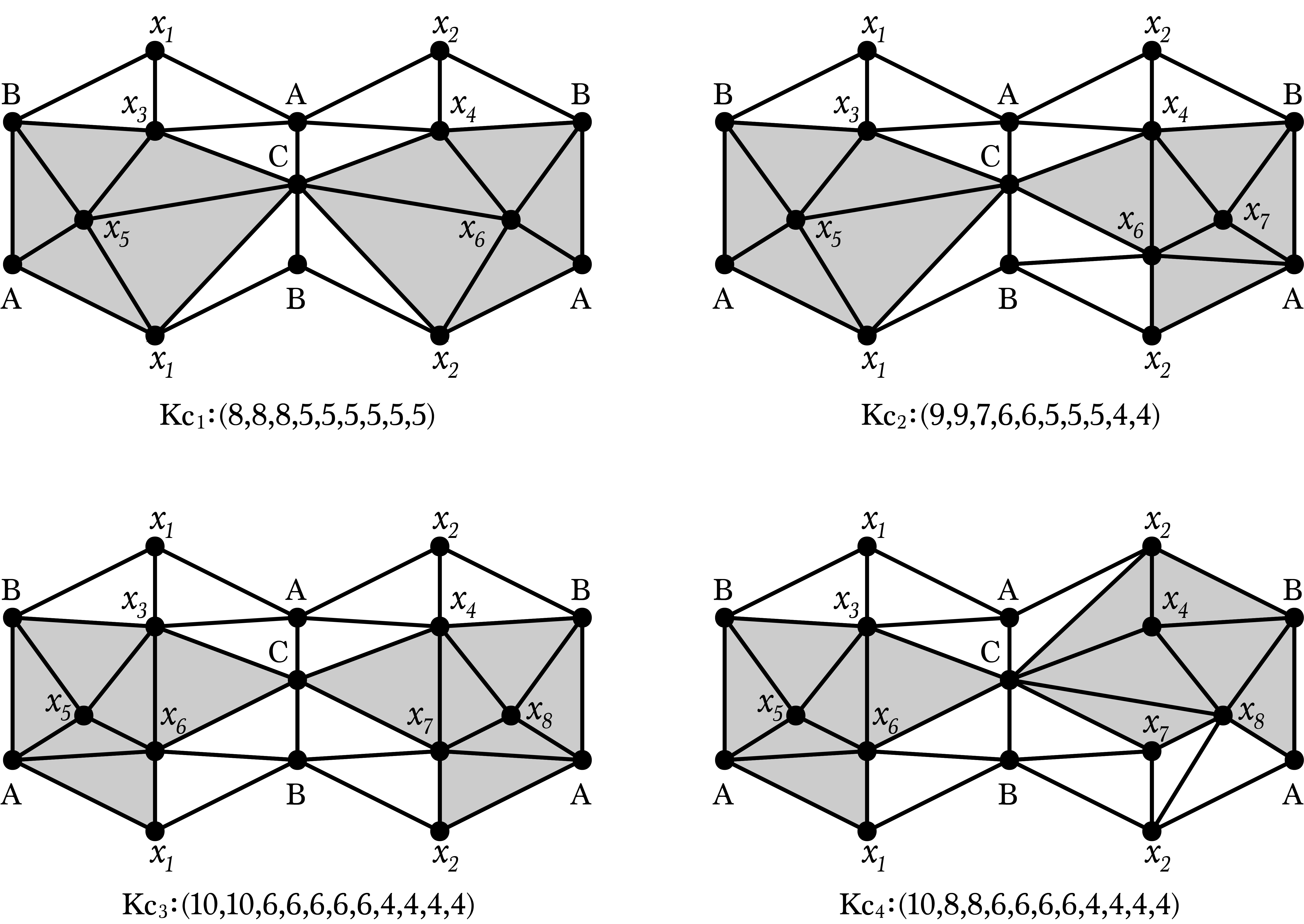}
\caption{Spanning pinched discs at $C$ in the cross-cap irreducible triangulations of $K$.}\label{fig:pinched_disc}
\end{figure}

\begin{figure}[H]
  \centering
\includegraphics[width=0.7\textwidth]{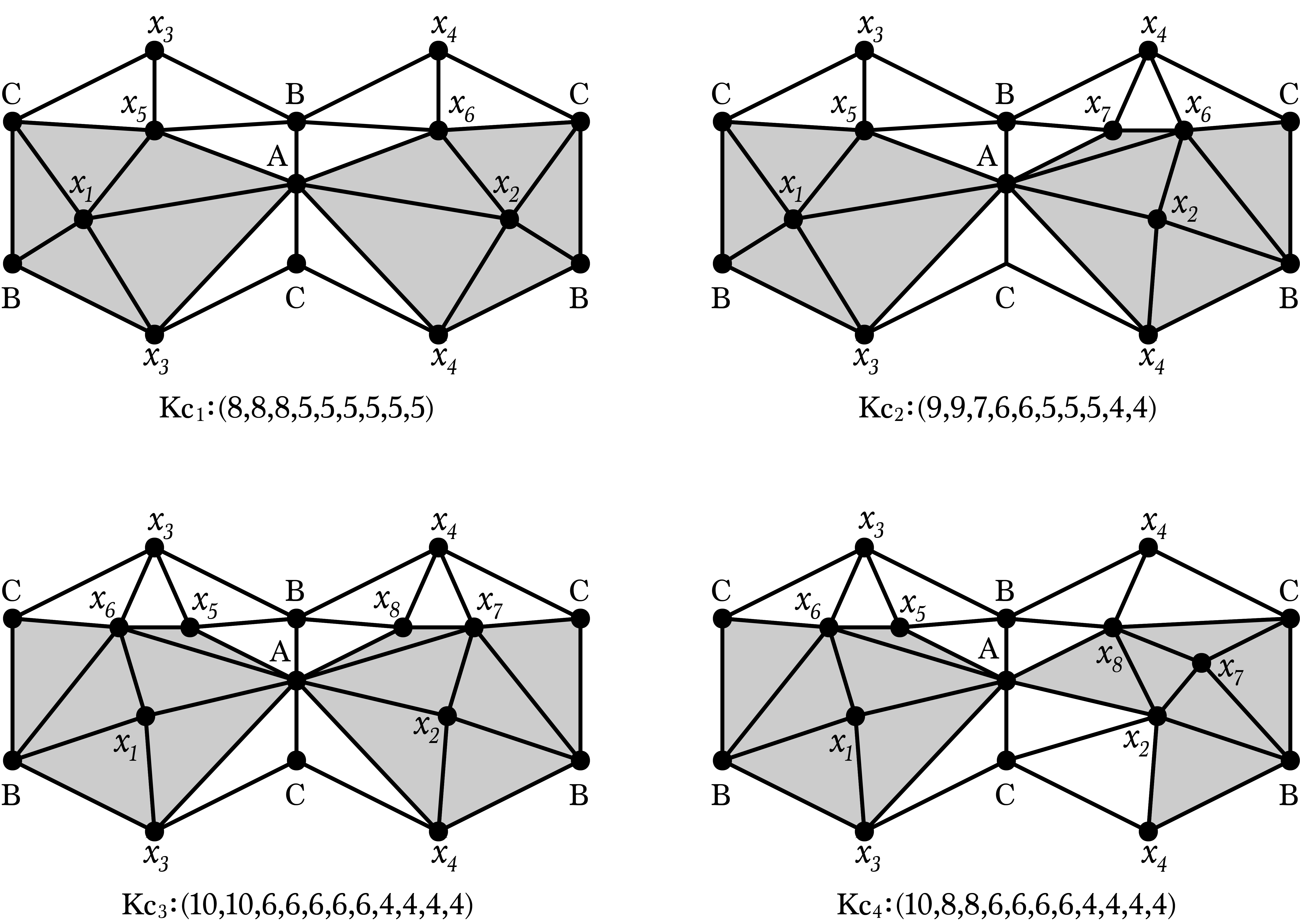}
\caption{Spanning pinched discs at $A$ in the cross-cap irreducible triangulations of $K$.}\label{fig:pinched_disc-A}
\end{figure}

\begin{figure}[H]
  \centering
\includegraphics[width=0.7\textwidth]{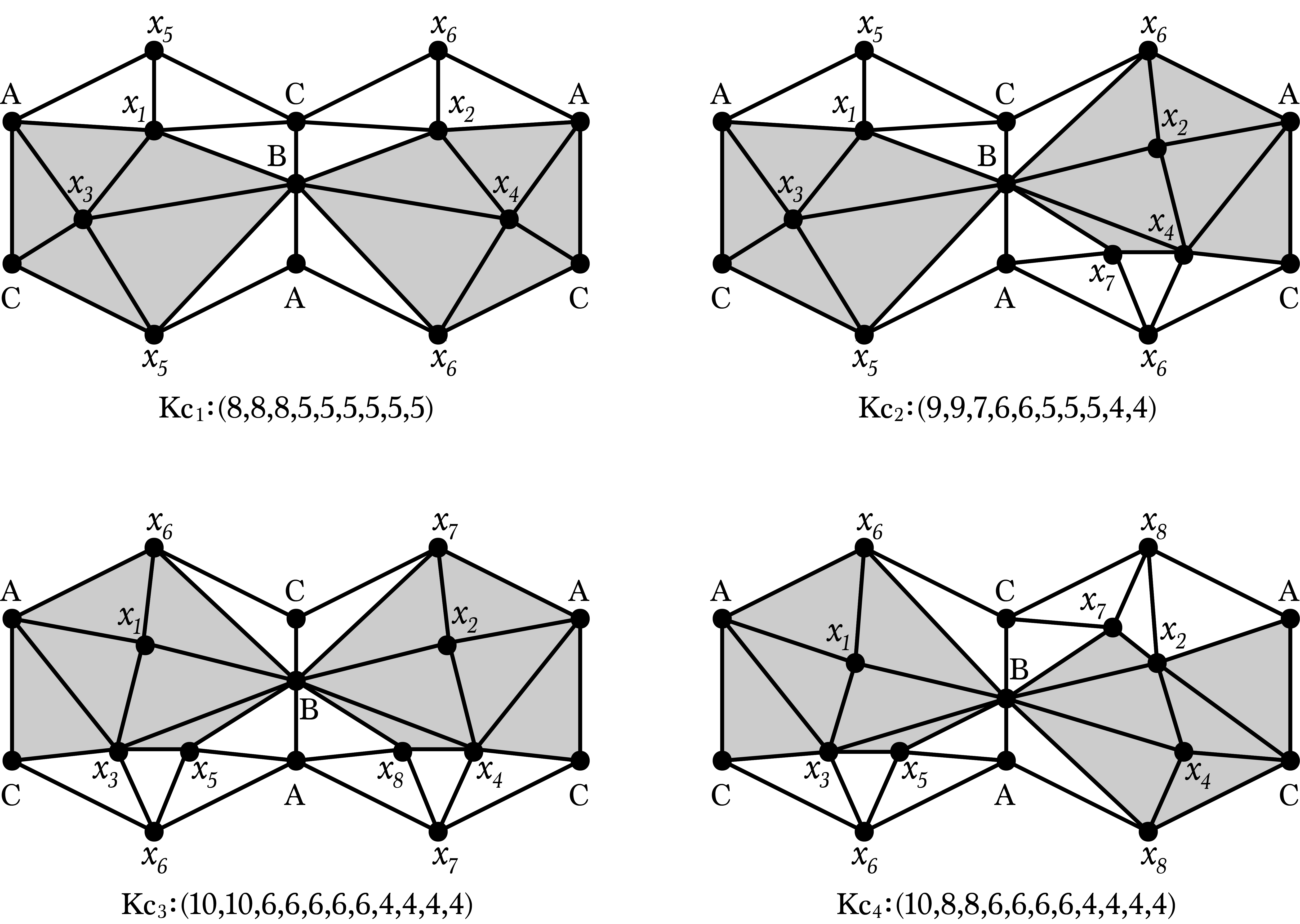}
\caption{Spanning pinched discs at $B$ in the cross-cap irreducible triangulations of $K$.}\label{fig:pinched_disc-B}
\end{figure}

After the first split we can resolve the singularity and choose $S'$ to be a spanning cylinder; see Figure~\ref{fig:resolve-singularity} for illustration.

\begin{figure}[H]
  \centering
\includegraphics[width=\textwidth]{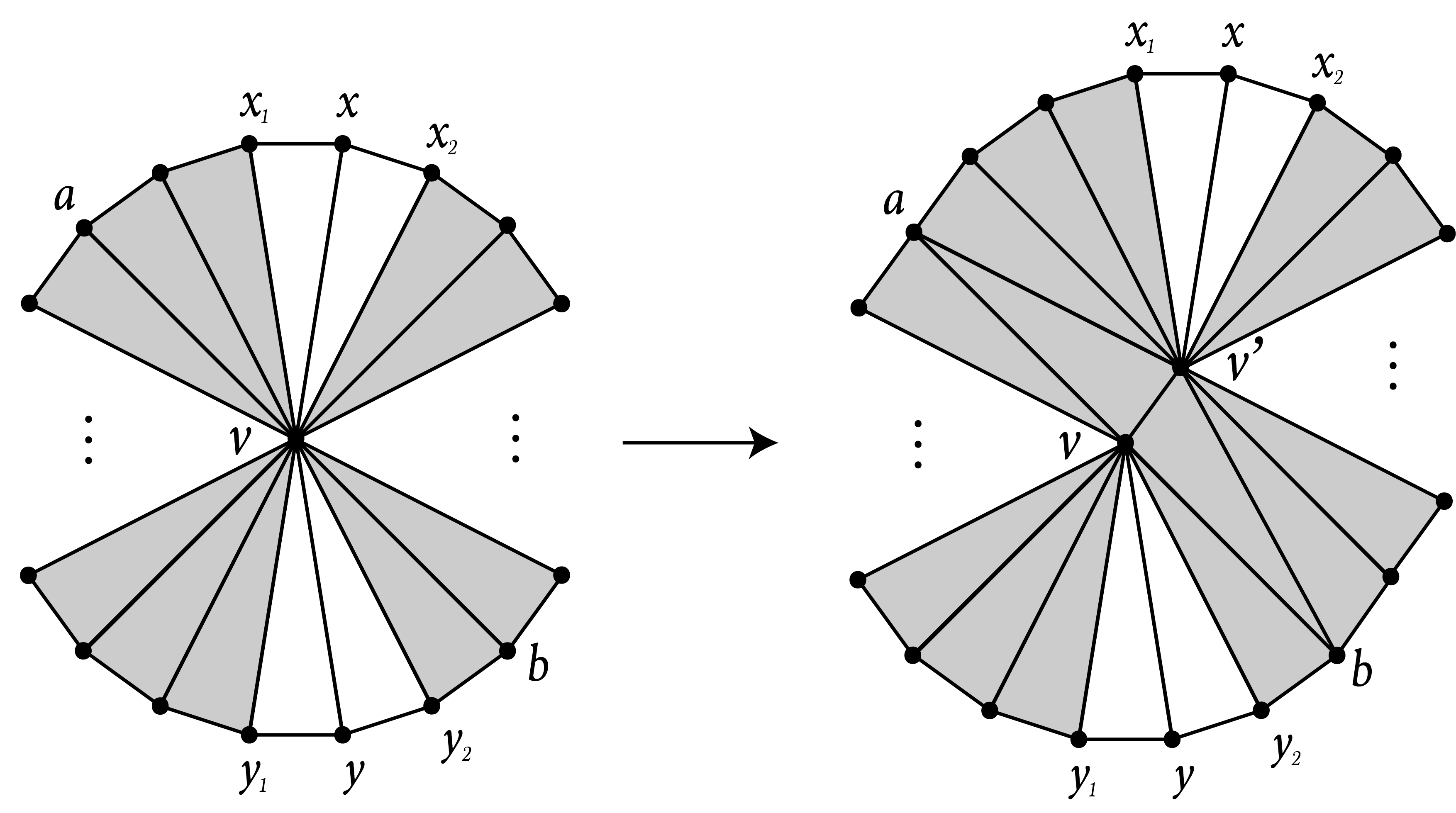}
\caption{Resolving a singularity: from a pinched disc to a cylinder.}\label{fig:resolve-singularity}
\end{figure}
Formally, in Proposition~\ref{prop:resolve_singularity} below we describe how to resolve a singularity in a spanning pinched surface $S$ under certain conditions; these conditions are met for the surfaces $S$ chosen in the irreducible cross-cap triangulations of the Klein bottle (see Figures~\ref{fig:pinched_disc},~\ref{fig:pinched_disc-A} and~\ref{fig:pinched_disc-B}).

The Extension Theorem~\ref{thm:extention} shows that further splits (after the first split and resolving the singularity) preserve admitting a spanning cylinder.
This completes the proof of the main Theorem~\ref{thm:Main}, modulo Propositions~\ref{prop:resolve_singularity} and~\ref{prop:ABCsplitsFirst}.
$\square$

Next we explain how to resolve a singularity.
Let $\Gamma$ be a $k$-cycle in the $1$-skeleton (graph) of $\Delta$. A vertex split on $\Delta$ is called \emph{$\Gamma$-elongating} if it induces a vertex split on $\Gamma$, thus, making it a $(k+1)$-cycle.

\begin{proposition}\label{prop:resolve_singularity}
Let $\Delta$ be a triangulated closed surface and let $S$ be vertex spanning
pinched surface with exactly one singular vertex $v$. Let $\Gamma\subseteq\Delta$ be a cycle, $v\in \Gamma$, such that $x$ and $y$ are the neighbors of $v$ in $\Gamma$. Suppose that $S$ contains all of the triangles in
the star of $v$ in $\Delta$ except for those which are incident with either of the edges $vx$ and $vy$, and that $S$ contains non of the edges in $\lk_v(\Delta)$ incident with either of $x$ and $y$. Then,

(1) Every $\Gamma$-elongating vertex split of $\Delta$ at $v$
induces a vertex split of $S$ at $v$ that yields an
extendible spanning simplicial surface $S'$ of the resulted complex $\Delta'$.

(2) If $S$ is a pinched disc then $S'$ from part (1) is a cylinder.
\end{proposition}
\begin{proof}
Denote the cycle $\lk_v(\Delta)$ by $(x,x_2,\ldots,y_2,y,y_1,\ldots,x_1,x)$ in cyclic order.
By assumption, the edges in the intersection of $S$ with the cycle $\lk_v(\Delta)$ are exactly those in the two disjoint intervals $[x_2,y_2]$ and $[y_1,x_1]$. In particular, as $S$ is a pinched surface, at $v$, not containing $vx$ nor $vy$, the six vertices $x,x_1,x_2,y,y_1,y_2$ are distinct.

Denote by $a$ and $b$ the two common neighbors of $v$ and $v'$ in $\Delta'$. As the vertex split is $\Gamma$-elongating, w.l.o.g. the path $(y,v,v',x)$ is contained in $\Delta'$ and $a\in [y_1,x_1]$ and $b\in[x_2,y_2]$.

Then the induced vertex split of $S$ at $v$ yields $S'$ with the following triangles containing $v$ or $v'$: $vv'a,vv'b$, $vst$ for consecutive vertices (in $\lk_v(\Delta)$)  $s,t\in[y_1,a]\cup [b,y_2]$, and $v'st$ for consecutive vertices (in $\lk_v(\Delta)$)  $s,t\in[x_2,b]\cup [a,x_1]$.

Then $S'$ is indeed extendible, namely $S'$ contains an edge from every triangle of $\Delta'$: the only triangles in $\Delta'$ containing $v$ or $v'$ that are not in $S'$ are $v'x_1x, v'xx_2,vy_2y$ and $vyy_1$; they contain the following edges of $S'$ resp.: $v'x_1,v'x_2,vy_2$ and $vy_1$.
To complete the proof of part (1) we verify that $S'$ is a surface with boundary -- indeed, the link of $v'$ (resp. $v$) in $S'$ is the path $(x_2,\ldots,b,v,a,\ldots,x_1)$ (resp. $(y_1,\ldots,a,v',b,\ldots,y_2)$), where the dots $\ldots$ refer to the cyclic order in $\lk_v(\Delta)$.

For part (2) note that $S'$ is homeomorphic to $S$ union a small disc with center at $v$ (in the realization space $|\Delta|$); for $S$ a pinched disc at $v$ this union is a cylinder.
\end{proof}

\section{Rearranging vertex splits}\label{sec:rearrange}
As promised, the following proposition reduces the treatment of triangulations of the Klein bottle obtained from one of the cross-cap triangulations, to the case where the first split makes the missing triangle $ABC$ a $4$-cycle.
\begin{proposition}\label{prop:ABCsplitsFirst}
Let $\Delta$ be obtained from one of the four cross-cap triangulations of $K$ by a sequence of $t$ 
vertex splits such that the missing triangle $ABC$ survived all the splits except for the last ($t$-th) split, which elongated it to an induced $4$-cycle, denote it by $ABCD$.

If $t\ge 2$ then there exists an edge contraction in $\Delta$, whose restriction to $ABCD$ is trivial, namely, non of the edges in $ABCD$ was contracted.
\end{proposition}
\begin{proof}
Let the $(t-1)$-th vertex split at $v$ introduce new vertex $v'$, changing a triangulation $\Delta''$ of $K$ to $\Delta'$, and let the $t$-th vertex split be at say $A$ (with $D=A'$, similarly for splits at $B$ or $C$; possibly $v=A$), changing $\Delta'$ to $\Delta$.
If the edge $v'v\in \Delta$ cannot be contracted, then $\Delta$ contains a missing triangle $v'vu$. However, this missing triangle was created by the $t$-th split, at $A$.

\textbf{Case $A\neq v$}.
Then in $\Delta'$ there exists the triangle $Av'v$, and $\Delta'$ splits at $A$ to create a missing triangle $A'v'v$ in $\Delta$ (recall our notation $D=A'$). Thus, $v$ and $v'$ are consecutive in $\lk_A(\Delta')$ and are the two common neighbors of $A$ and $A'$ in $\Delta$. However, as $ABC$ does not survive the $t$-th split, $v$ and $v'$ must separate $B$ and $C$ in the cycle $\lk_A(\Delta')$, a contradiction.

\textbf{Case $A=v$}. Then, similarly, there exists a triangle $Av'u \in \Delta'$, and for the $t$-th split at $A$, $v'$ and $u$ are the two common neighbors of $A$ and $A'$ in $\Delta$, contradicting that $v'$ and $u$ must separate $B$ and $C$ in the cycle $\lk_A(\Delta')$.
\end{proof}

Iterating the edge contractions guaranteed in Proposition~\ref{prop:ABCsplitsFirst} we would either reach a non cross-cap irreducible triangulation, or end up with a triangulation that is obtained from one of the cross-cap triangulations  via a single vertex split, that elongates $ABC$ to a $4$-cycle; as required to conclude the proof of the main Theorem~\ref{thm:Main}. $\square$

\begin{remark}
On a combinatorial level, edge contractions always \emph{do} commute; however, reordering them may \emph{not} preserve the topology of the complex.
\end{remark}
More formally, let us first set a notational convention: let $\Delta$ be a simplicial complex on the vertex set $V=\{v_{\{1\}},\ldots,v_{\{n\}}\}$, and $\Delta'$ is obtained from $\Delta$ by a sequence of edge contractions.
Along this sequence, when we contract the edge $v_Sv_T$ we name the ``merged" vertex by $v_{S\cup T}$ while the other vertices $v_P\neq v_S,v_T$ keep their names. Note that $S$ and $T$ are disjoint subsets of $[n]$.
\begin{observation}[Commutativity of edge contractions]\label{obs:edge_contractions_commute}
Under the above convention:

(0) A subset $\{v_{T_1},\ldots,v_{T_m}\}$ of vertices in $\Delta'$ is a face of $\Delta'$ iff there exist indices $i_j\in T_j$ such that $\{v_{i_1},\ldots,v_{i_m}\}$ is a face in $\Delta$. Thus,

(i) If each of $\Delta'$ and $\Delta''$ is obtained from $\Delta$ by some sequence of edge contractions, and the names of vertices are identical in $\Delta'$ and $\Delta''$ then $\Delta'=\Delta''$.

(ii) If $\Delta'$ is obtained from $\Delta$ by a sequence of edge contractions, and the edge $v_{\{i\}}v_{\{j\}}\in \Delta$ satisfies that $i,j\in T$ for some vertex $v_T\in \Delta'$, then there exists another sequence of edge contractions that starts with $\Delta$, ends with $\Delta'$, and
the last edge contracted is of the form $v_{I}v_{J}$ such that $i\in I$, $j\in J$ (and $T=I\cup J$).
\end{observation}
\begin{proof}
(0) is clear for a single edge contraction; in that case at most one of the $T_l$s has size two and all others are singletons. For the general case we iterate, i.e. induct on the number of edge contractions in the sequence yielding from $\Delta$ to $\Delta'$.

(i) follows at once from (0).


(ii) can be proved directly, by showing how to modify the sequence of edge contractions from $\Delta$ to $\Delta'$. To avoid extra notation, note that by (i) it is enough to show the claim in (ii) for the one skeleta of $\Delta$ and $\Delta'$; for this case it is a basic fact on graph minors.
\end{proof}

However, if we care about preserving the topology, or even just the homology, edge contractions may not commute. For example, start with the boundary of a tetrahedron on the vertex set $\{1,2,3,4\}$, and iteratively at the $t$-th vertex split perform a stellar subdivision at the triangle $\{1,2,t+3\}$ by a new vertex $t+4$ (this splits vertex $t+3$ say, and introduces a new missing triangle $\{1,2,t+3\}$). The resulted complex is a stacked sphere. Pick $t\ge 5$. If we contract edges so that $4,5,\ldots, t+3$ are identified we obtain a $2$-sphere, the boundary of a tetrahedron. However, if we contract the edge $45$ first we obtain two $2$-spheres glued along an edge.


\section{Concluding remarks}\label{sec:conclude}
For higher genus $g$, the list of irreducible triangulations of $M_g$ is not known, so the approach taken here is not applicable. Adding the empty triangles in an irreducible triangulation (every edge is contained in an empty triangle!) gives more flexibility in finding a strongly connected spanning subcomplex. This approach may be useful towards Problem~\ref{prob:spanning}.

%
%
\textbf{Acknowledgements.}
An extended abstract to this paper will be presented at FPSAC2022~\cite{Simion-FPSAC}.
We thank the anonymous referees of FPSAC2022
for their helpful feedback.

\bibliographystyle{plain}
\bibliography{Spanning-Laman}
\end{document}